\documentclass[a4paper]{amsart}
\usepackage{textcomp, amsmath, amssymb, amsthm}
\usepackage{graphicx}
\usepackage{enumerate}
\usepackage[margin=1in]{geometry}
\usepackage{stmaryrd}
\usepackage{physics}
\usepackage[colorlinks,
linkcolor=black,
anchorcolor=black,
citecolor=black,
urlcolor=black
]{hyperref}

\usepackage{tikz} 
\usepackage[foot]{amsaddr}

\tikzset{every picture/.style={line width=0.75pt}} 

\newtheorem{theorem}{Theorem}[section]
\newtheorem{lemma}[theorem]{Lemma}
\newtheorem{corollary}[theorem]{Corollary}
\newtheorem{example}[theorem]{Example}

\newtheorem{proposition}[theorem]{Proposition}

\theoremstyle{definition} 
\newtheorem{remark}[theorem]{Remark}
\newtheorem{defin}[theorem]{Definition}

\title{Some Computational Results on Koszul-Vinberg Cochain Complexes}

\author{Hanwen Liu $^{1,2}$}
\thanks{$^1$ Corresponding author.}
\thanks{$^2$ Mathematics Institute, University of Warwick, Coventry, CV4 7AL, UK. \textsl{Email: hanwen.liu@warwick.ac.uk.}}

\author{Jun Zhang $^{3,4}$}
\thanks{$^3$ University of Michigan, Ann Arbor, MI 48109, USA. \textsl{Email: junz@umich.edu}.}
\thanks{$^4$ Shanghai Institute for Mathematics and Interdisciplinary Sciences, Shanghai, China.  \textsl{Email: junz@simis.cn}.}

\begin{document}

\begin{abstract}
    An affine connection is said to be flat if its curvature tensor vanishes identically. Koszul-Vinberg (KV for abbreviation) cohomology has been invoked to study the deformation theory of flat and torsion-free affine connections on tangent bundle. In this Note, we compute explicitly the differentials of various specific KV cochains, and study their relation to classical objects in information geometry, including deformations associated with projective and dual-projective transformations of a flat and torsion-free affine connection. As an application, we also give a simple yet non-trivial example of a KV algebra of which second cohomology group does not vanish.
\end{abstract}

\maketitle

\section{Introduction and Backgrounds}
A differentiable manifold $M$ is called a locally flat manifold if it admits a flat and torsion-free connection $\nabla$ on its tangent bundle $TM\rightarrow M$, and in this case the connection $\nabla$ is also termed a locally flat structure on $M$. In information geometry, a locally flat structure $\nabla$ along with a Riemannian metric $g$ gives rise to another flat conjugate connection $\nabla^*$. The metric $g$ becomes a Hessian metric if and only if the conjugate connection is also torsion-free. 

The deformation of locally flat structures has been studied by means of the theory of 
Koszul-Vinberg cohomology \cite{1}. Boyom and Byande applied the theory of KV cohomology to deformation of locally flat structures. In \cite{11} they relate symmetric zeros of Maurer-Cartan polynomial map to the deformation theory of KV algebra of a flat torsion-free connection $\nabla$. As KV cohomology is so far the main algebraic topological tool that is utilized in information geometry, further understanding of its properties seems of interest.  

Our note investigates KV cohomology by explicitly calculate KV differential of specific geometric objects.
After a brief review of KV algebra of a flat torsion-free connection and its cohomology, we provide a few results related to the first cochain group and the second cochain group of KV cohomology.  As specific applications, characterizations of projective transformation and dual-projective transformation to a flat and torsion-free connection are both described in terms of a vanishing conndition on KV differential of their associated 2-cochains. Finally, to show that KV cohomology is not vaccuous, we construct an example of KV cohomology with non-vanishing second cochain group for the simplest case of a planar domain with its usual Euclidean metric. 

The Appendix provided a review of exterior covariant derivative and de Rham cohomology twisted by a local system, to allow uninitiated readers a comparison with and appreciation of KV cohomology.

\section{Brief Review of Koszul-Vinberg Cohomology}
Let $M$ be an arbitrary but fixed smooth manifold, and let
$$
\nabla\colon \Gamma(T M) \times \Gamma(T M) \longrightarrow \Gamma(T M), \quad(X, Y) \longmapsto \nabla_X Y
$$
be a flat torsion-free affine connection on the tangent bundle $T M \rightarrow M$, where $\Gamma(T M)$ is the $\mathbb{R}$-vector space of $C^{\infty}$ vector fields on $M$.
\begin{defin}\label{defin_2_1}
    The $\mathbb{R}$-algebra $(\Gamma(T M), \nabla)$ is said to be the KV algebra of $\nabla$.
\end{defin}

To simplify the notation, throughout this article we denote $A:=(\Gamma(T M), \nabla)$.

Recall that a vector field $Z \in \Gamma(T M)$ is termed a Jacobi element of $A$, if $\nabla_X \nabla_Y Z=\nabla_{\nabla_X Y} Z$ holds for all $X, Y \in \Gamma(T M)$.

\begin{defin}\label{defin_2_3}
    Define $C^0(A)$ to be the subspace of Jacobi elements of $A$, and for each integer $n \geqslant 1$ define $C^n(A):=\operatorname{Hom}_{\mathbb{R}}\left(\Gamma(T M)^{\otimes n}, \Gamma(T M)\right)$, or equivalently, the collection of multilinear maps from the $n$-fold Cartesian product $\Gamma(T M) \times \cdots \times \Gamma(T M)$ to $\Gamma(T M)$.
\end{defin}

In addition, we shall introduce the following important notation:

For each $X \in \Gamma(T M)$ and each $\theta \in C^n(A)$ with $n \geqslant 1$, let $\nabla_X \theta$ be the element of $C^n(A)$ satisfying $(\nabla_X \theta)\left(X_1, \ldots, X_n\right):=\nabla_X\left(\theta\left(X_1, \ldots, X_n\right)\right)-\left[\theta\left(\nabla_X X_1, X_2, \ldots, X_n\right)+\cdots+\theta\left(X_1, \ldots, X_{n-1}, \nabla_X X_n\right)\right]$.

For any integer $n \geqslant 1$, notice that if $\theta \in C^n(A)$ is a $(1, n)$-tensor on $M$ then the notation above recovers the Leibniz rule.
\begin{defin}\label{defin_2_5}
    Define endomorphism $d_{\mathrm{KV}}\colon\bigoplus_{i=0}^{\infty} C^i(A) \longrightarrow \bigoplus_{i=0}^{\infty} C^i(A)$ of degree +1 of the graded $\mathbb{R}$-vector space $\bigoplus_{i=0}^{\infty} C^i(A)$ as it follows:
    \begin{enumerate}
        \item[$(i)$] for each $X \in C^0(A)$, define $d_{\mathrm{KV}} X \in C^1(A)$ by $\left(d_{\mathrm{KV}} X\right)(Y):=[X, Y]$;
        \item[$(ii)$] for each $\theta \in C^n(A)$ with $n \geqslant 1$, define $d_{\mathrm{KV}} \theta \in C^{n+1}(A)$ by
        $$
        \left(d_{\mathrm{KV}} \theta\right)\left(X_1, \ldots, X_{n+1}\right):=\sum_{i=1}^n(-1)^i\left[\left(\nabla_{X_i} \theta\right)\left(X_1, \ldots, \hat{X}_i, \ldots, X_{n+1}\right)+\nabla_{\theta\left(X_1, \ldots, \hat{X}_i, \ldots, X_n, X_i\right)} X_{n+1}\right]
        $$
        where the hat on $\hat{X}_i$ indicates that the $X_i$ term is omitted.
    \end{enumerate}
\end{defin}
\begin{lemma}\label{lemma_2_6}
    It holds that $d_{\mathrm{KV}}{ }\circ d_{\mathrm{KV}}=0$.
\end{lemma}
\begin{proof}
    See for example \cite{1} for a detailed proof.
\end{proof}
\begin{defin}\label{defin_2_7}
    The differential graded $\mathbb{R}$-vector space $\left(\bigoplus_{i=0}^{\infty} C^i(A), d_{\mathrm{KV}}\right)$ is said to be the KV cochain complex of $A$, and the elements of $C^n(A)$ are termed KV $n$-cochains.
\end{defin}

We shall now recall a classical application of the cohomology theory of KV cochain complexes \cite{11}. The following definitions and theorem are from their work. 
\begin{defin}\label{defin_2_8}
    A family $\left\{\nabla^t \in C^2(A) \mid t \in \mathbb{R}\right\}$ of flat torsion-free affine connections on the tangent bundle $T M \rightarrow M$ is said to be a smooth deformation of $\nabla$, if $\nabla^0=\nabla$ and for all $X, Y \in \Gamma(T M)$ the mapping
    $$
    M \times \mathbb{R} \longrightarrow T M, \quad(p, t) \longmapsto\left(\nabla_X^t Y\right)(p)
    $$
    is smooth.
\end{defin}
\begin{defin}\label{defin_2_9}
    A smooth deformation $\left\{\nabla^t \in C^2(A) \mid t \in \mathbb{R}\right\}$ is said to be trivial, if there exists a one-parameter subgroup
    $$
    \phi\colon \mathbb{R} \longrightarrow \operatorname{Diff}(M), \quad t \longmapsto \phi^t
    $$
    of the group of diffeomorphisms of $M$, such that
    $$
    \nabla_X^t Y=d \phi^t\left(\nabla_{d \phi^{-t}(X)} d \phi^{-t}(Y)\right)
    $$
    for all $X, Y \in \Gamma(T M)$ and $t \in \mathbb{R}$.
\end{defin}

In \cite{1}, the following rigidity theorem is proved.
\begin{theorem}\label{theorem_2_10}
    Suppose that $M$ is compact. If the second cohomology group of $\left(\bigoplus_{i=0}^{\infty} C^i(A), d_{\mathrm{KV}}\right)$ vanishes, then all the smooth deformations of $\nabla$ are trivial.
\end{theorem}

For more applications, we refer to \cite{2},\cite{3}, and \cite{4}.

\section{Results on the First Cochain Group}

Let $\theta \in C^1(A)=\operatorname{End}_\mathbb{R}(\Gamma(TM))$ be an arbitrary $\mathrm{KV}$ $1$-cochain of $A$. Then by the definition of $d_\mathrm{KV}$, we have
$$
(d_\mathrm{KV} \theta)(X, Y)=-\nabla_X \theta(Y)+\theta\left(\nabla_X Y\right)-\nabla_{\theta(X)} Y. 
$$

\begin{example}\label{example_3_1}
    Let $f \in C^{\infty}(M)$ be a fixed smooth function, and define a $\mathrm{KV}$ $1$-cochain $\theta \in C^1(A)$ by
    $$
    \theta\colon \Gamma(T M) \longrightarrow \Gamma(T M), \quad Z \longmapsto f Z.
    $$
    Then for any $X, Y \in \Gamma(T M)$, we have that $\left(d_{\mathrm{KV}} \theta\right)(X, Y)=-\nabla_X(f Y)+f \nabla_X Y-\nabla_{f X} Y=-\nabla_X(f Y)$.
\end{example}

We identify $\Gamma(T M)$ with the module of derivations of the $\mathbb{R}$-algebra $C^{\infty}(M)$. Recall that the Poisson bracket $[X, Y] \in \Gamma(T M)$ of two vector fields $X$ and $Y$ on $M$ is defined by
$$
[X, Y]\colon C^{\infty}(M) \longrightarrow C^{\infty}(M), \quad f \longmapsto X(Y f)-Y(X f)
$$
and $\mathfrak{X}:=(\Gamma(T M),[\cdot, \cdot])$ is a Lie algebra. Denote by $\mathrm{ad}\colon\mathfrak{X} \longrightarrow \operatorname{Der}(\mathfrak{X})$ the adjoint representation of $\mathfrak{X}$.

\begin{example}\label{example_3_2}
  Let $Z\in \Gamma(TM)$ be a fixed vector field and define $\theta:=\mathrm{ad}_Z\in C^1(A)$. Then, since $\nabla$ is flat and torsion-free, we have that 
  $$
    \begin{aligned}
    \left(d_{\mathrm{KV}} \theta\right)(X, Y) & =-\nabla_X\left(\mathrm{ad}_Z(Y)\right)+\mathrm{ad}_Z\left(\nabla_X Y\right)-\nabla_{\mathrm{ad}_Z(X)} Y \\
    & =\nabla_X[Y, Z]-\left[\nabla_X Y, Z\right]+\nabla_{[X, Z]}Y \\
    & =\nabla_X\left(\nabla_Y Z-\nabla_Z Y\right)-\nabla_{\nabla_XY} Z+\nabla_Z \nabla_X Y+\nabla_{[X, Z]} Y \\
    & =\nabla_X \nabla_Y Z-\nabla_{\nabla_X Y}Z-\left(\left[\nabla_X, \nabla_Z\right] Y-\nabla_{[X, Z]} Y\right) \\
    & =\nabla_X \nabla_Y Z-\nabla_{\nabla_X Y}Z
    \end{aligned}
    $$
    holds for all $X, Y \in \Gamma(T M)$. In particular, $d_{\mathrm{KV}} \theta$ is a $(1,2)$-tensor.
\end{example}

\begin{theorem}\label{theorem_3_3}
    Let $\theta \in C^1(A)$ be a $\mathrm{KV}$ $1$-cochain. Then the following statements are equivalent:
    \begin{enumerate}
        \item[$(i)$] $\left(d_{\mathrm{KV}} \theta\right)(X, Y)=\left(d_{\mathrm{KV}} \theta\right)(Y, X)$ for all $X, Y \in \Gamma(T M)$;
        \item[$(ii)$] there exists $X \in \Gamma(T M)$ such that $\theta=\mathrm{ad}_X$.
    \end{enumerate}
\end{theorem}
\begin{proof}
    We first prove that $(ii)$ implies $(i)$. Suppose that there exists $Z \in \Gamma(T M)$ such that $\theta=\mathrm{ad}_Z$, then by Example \ref{example_3_2} we have that
    $$
    \left(d_{\mathrm{KV}} \theta\right)(X, Y)=\nabla_X \nabla_Y Z-\nabla_{\nabla_X Y}Z
    $$
    holds for all $X, Y \in \Gamma(T M)$. Since $\nabla$ is flat, for any $X, Y \in \Gamma(T M)$ we have 
    $$
    \nabla_X \nabla_Y-\nabla_{\nabla_XY}-\nabla_Y \nabla_X+\nabla_{\nabla_YX}=\left[\nabla_X, \nabla_Y\right]-\nabla_{[X, Y]}=0
    $$
    and hence $\nabla_X \nabla_Y Z-\nabla_{\nabla_X Y} Z=\nabla_Y \nabla_X Z-\nabla_{\nabla_Y X} Z$, i.e. $\left(d_{\mathrm{KV}} \theta\right)(X, Y)=\left(d_{\mathrm{KV}} \theta\right)(Y, X)$.

    Now we shall prove that $(i)$ implies $(ii)$. Since $\nabla$ is torsion free, we have that
    $$
    \begin{aligned}
    \left(d_{\mathrm{KV}} \theta\right)(X, Y) & =-\nabla_X(\theta(Y))+\theta\left(\nabla_X Y\right)-\nabla_{\theta(X)} Y \\
    & =\theta\left(\nabla_X Y\right)-[\theta(X), Y]-\left(\nabla_X(\theta(Y))+\nabla_Y(\theta(X))\right)
    \end{aligned}
    $$
    holds for all $X, Y \in \Gamma(T M)$. Suppose that $(i)$ is satisfied, then for any $X, Y \in \Gamma(T M)$, we have
    $$
    \begin{aligned}
    \theta([X, Y])-([\theta(X), Y]+[X, \theta(Y)]) & =\theta\left(\nabla_X Y\right)-[\theta(X), Y]-\left(\theta\left(\nabla_Y X\right)-[\theta(Y), X]\right) \\
    & =\left(d_{\mathrm{KV}} \theta\right)(X, Y)-\left(d_{\mathrm{KV}} \theta\right)(Y, X) \\
    & =0
    \end{aligned}
    $$
    i.e. $\theta([X, Y])=[\theta(X), Y]+[X, \theta(Y)]$. Therefore we conclude that $\theta \in \operatorname{Der}(\mathfrak{X})$. Since by \cite{5} all derivations of $\mathfrak{X}$ are inner, there exists $Z \in \Gamma(T M)$ such that $\theta=\mathrm{ad}_Z$.
\end{proof}
\begin{remark}\label{remark_3_4}
    For an arbitrary $\theta \in C^1(A)$, the equality $\left(d_{\mathrm{KV}} \theta\right)(X, Y)=\left(d_{\mathrm{KV}} \theta\right)(Y, X)$ does not hold for all $X, Y \in \Gamma(T M)$ in general. For example, if $\theta \in C^1(A)$ is the identity map, then by Example \ref{example_3_1} $\left(d_{\mathrm{KV}} \theta\right)(X, Y)=-\nabla_X Y$ and hence $\left(d_{\mathrm{KV}} \theta\right)(X, Y)-\left(d_{\mathrm{KV}} \theta\right)(Y, X)=[Y, X]$.
\end{remark}

\section{Results on the Second Cochain Group}

Let $\theta \in C^2(A)$ be an arbitrary $\mathrm{KV}$ $2$-cochain of $A$. Then by the definition of $d_\mathrm{KV}$, we have
$$
\begin{aligned}
    \left(\mathrm{d}_{K V} \theta\right)(X, Y, Z)=&-\nabla_X \theta(Y, Z)+\theta\left(\nabla_X Y, Z\right)+\theta\left(Y, \nabla_X Z\right)-\nabla_{\theta(Y, X)} Z\\ 
    &+\nabla_Y \theta(X, Z)-\theta\left(\nabla_Y X, Z\right)-\theta\left(X, \nabla_Y Z\right)+\nabla_{\theta(X, Y)} Z\\ 
    =&(\nabla_Y \theta)(X, Z)-(\nabla_X \theta)(Y, Z)+\nabla_{\theta(X, Y)-\theta(Y, X)} Z.
\end{aligned}
$$

\begin{proposition}\label{proposition_4_1}
    Let $\operatorname{Id}:\Gamma(TM)\to\Gamma(TM)$ be the identity map. Then it holds that $d_{\mathrm{KV}}(-\operatorname{Id})=\nabla$. In particular $d_{\mathrm{KV}} \nabla=0$.
\end{proposition}
\begin{proof}
    The first assertion follows from Example \ref{example_3_1}. By Lemma \ref{lemma_2_6}, $d_{\mathrm{KV}} \nabla=d_{\mathrm{KV}} d_{\mathrm{KV}} (-\operatorname{Id})=0$.
\end{proof}
Note that the sign convention in $d_{\mathrm{KV}} (\operatorname{Id})=-\nabla$ follows Boyom.

\begin{theorem}\label{theorem_4_2}
    Let $D$ be a torsion-free affine connection on the tangent bundle $T M \rightarrow M$, and let $\theta:=\nabla-D \in C^2(A)$. Then the following properties hold:
    \begin{enumerate}
        \item[$(i)$] $\theta$ is a $(1,2)$-tensor;
        \item[$(ii)$] $\theta(X, Y)=\theta(Y, X)$ for all $X, Y \in \Gamma(T M)$;
        \item[$(iii)$] $\left(d_{\mathrm{KV}} \theta\right)(X, Y, Z)=\left(\nabla_Y \theta\right)(X, Z)-\left(\nabla_X \theta\right)(Y, Z)$ for all $X, Y, Z \in \Gamma(T M)$, in particular, $d_{\mathrm{KV}} \theta$ is a $(1,3)$-tensor;
        \item[$(iv)$] $\theta \in \mathrm{Im}(d_{\mathrm{KV}})$ if and only if there exists $Z\in\Gamma(TM)$ such that 
        $$
        \theta(X,Y)=\nabla_X \nabla_Y Z-\nabla_{\nabla_X Y}Z
        $$
        for all $ X,Y\in\Gamma(TM)$.
    \end{enumerate}
\end{theorem}
\begin{proof}
    Properties $(i)$ and $(ii)$ are proved in \cite{6}. Property $(iii)$ follows from $(ii)$ and the very definition of $d_{\mathrm{KV}}$. Property $(iv)$ is a direct consequence of Example \ref{example_3_2} and Theorem \ref{theorem_3_3}.
\end{proof}

As special cases of Theorem \ref{theorem_4_2}, we now consider several concrete deformation of flat torsion-free connection. 

\begin{defin}\label{defin_4_3}
    The connection $D$ on tangent bundle $TM$ defined by
    $$
    D_X Y := \nabla_X Y + \omega(X) Y+\omega(Y) X
    $$
    where $\omega$ is a given 1-form, is said to be a projective transformation of $\nabla$.
\end{defin}

It is routine to check that the formula above indeed defines a connection.

\begin{proposition}\label{proposition_4_3}
    Suppose that $\operatorname{dim} M \geqslant 3$ and let $\omega \in \Omega^1(M)$ be a 1 -form. Let $\theta \in C^2(A)$ be the $\mathrm{KV}$ $2$-cochain satisfying $\theta(X, Y)=\omega(X) Y+\omega(Y) X$ for all $X, Y \in \Gamma(T M)$. Then $d_{\mathrm{KV}} \theta=0$ if and only if $\nabla \omega=0$.
\end{proposition}
\begin{proof}
    Take any arbitrary $X, Y, Z \in \Gamma(T M)$. Since $\theta(X, Y)=\theta(Y, X)$ and
    $$
    \begin{aligned}
    \left(\nabla_X \theta\right)(Y, Z)= & \nabla_X(\omega(Y) Z+\omega(Z) Y)-\omega\left(\nabla_X Y\right) Z \\
    & -\omega(Z) \nabla_X Y-\omega(Y) \nabla_X Z-\omega\left(\nabla_X Z\right) Y \\
    = & \nabla_X(\omega(Y) Z)-\omega\left(\nabla_X Y\right) Z-\omega(Y) \nabla_X Z \\
    & +\nabla_X(\omega(Z) Y)-\omega\left(\nabla_X Z\right) Y-\omega(Z) \nabla_X Y \\
    = & \left(\nabla_X \omega\right)(Y) Z+\left(\nabla_X \omega\right)(Z) Y,
    \end{aligned}
    $$
    we obtain that
    $$
    \begin{aligned}
    \left(d_{\mathrm{KV}} \theta\right)(X, Y, Z) & =\left(\nabla_Y \theta\right)(X, Z)-\left(\nabla_X \theta\right)(Y, Z)+\nabla_{\theta(X, Y)-\theta(Y, X)}Z \\
    & =\left(\nabla_Y \omega\right)(X) Z+\left(\nabla_Y \omega\right)(Z) X-\left(\nabla_X \omega\right)(Y) Z-\left(\nabla_X \omega\right)(Z) Y .
    \end{aligned}
    $$
    Therefore $\nabla \omega=0$ implies that $d_{\mathrm{KV}} \theta=0$.

    Now assume that $d_{\mathrm{KV}} \theta=0$. Let $e_1, \ldots, e_n$ be a local frame in the tangent bundle $T M \rightarrow M$. Take any arbitrary $i, j \in\{1, \ldots, n\}$. Since $n=\operatorname{dim} M \geqslant 3$, there exists $k \in\{1, \ldots, n\}$ such that $k \neq i$ and $k \neq j$. Since
    $$
    \left(\nabla_{e_k} \omega\right)\left(e_i\right) e_j+\left(\nabla_{e_k} \omega\right)\left(e_j\right) e_i-\left(\nabla_{e_i} \omega\right)\left(e_k\right) e_j-\left(\nabla_{e_i} \omega\right)\left(e_j\right) e_k
    =0,
    $$
    in particular we have $\left(\nabla_{e_i} \omega\right)\left(e_j\right)=0$. Therefore we conclude that $\nabla \omega=0$.
\end{proof}

\begin{defin}
    The connection $D$ on tangent bundle $TM$ defined by
    $$
    D_X Y := \nabla_X Y -h(X, Y) V 
    $$
    where $h$ is a given pseudo-Riemannian metric and $V$ is a given vector field, is said to be a dual-projective transformation of $\nabla$.  
\end{defin}

As in \ref{defin_4_3}, one can check that the formula above indeed defines a connection.

\begin{defin}\label{defin_4_4}
    A pseudo-Riemannian metric $h$ on $M$ is said to be Codazzi-coupled with $\nabla$, if the Codazzi equation
    $$
    \left(\nabla_X h\right)(Y, Z)=\left(\nabla_Y h\right)(X, Z)
    $$
    holds for all $X, Y, Z \in \Gamma(T M)$.
\end{defin}

\begin{proposition}\label{proposition_4_5}
    Let $h$ be a pseudo-Riemannian metric on $M$ and let $V \in \Gamma(T M)$ be a non-vanishing vector field parallel to $\nabla$. Let $\theta \in C^2(A)$ be the $\mathrm{KV}$ $2$-cochain satisfying $\theta(X, Y)=-h(X, Y) V$ for all $X, Y \in \Gamma(T M)$. Then $d_{\mathrm{KV}} \theta=0$ if and only if $h$ is Codazzi-coupled with $\nabla$.
\end{proposition}
\begin{proof}
    Take any arbitrary $X, Y, Z \in \Gamma(T M)$. Since $\theta(X, Y)=\theta(Y, X)$ and
    $$
    \begin{aligned}
    \left(\nabla_X \theta\right)(Y, Z) & =-\nabla_X(h(Y, Z) V)+h\left(\nabla_X Y, Z\right) V+h\left(Y, \nabla_X Z\right) V \\
    & =-\left(\nabla_X h\right)(Y, Z) V-h(Y, Z) \nabla_X V \\
    & =-\left(\nabla_X h\right)(Y, Z) V,
    \end{aligned}
    $$
    we obtain that
    $$
    \begin{aligned}
    \left(d_{\mathrm{KV}} \theta\right)(X, Y, Z) & =\left(\nabla_Y \theta\right)(X, Z)-\left(\nabla_X \theta\right)(Y, Z)+\nabla_{\theta(X, Y)-\theta(Y, X)}Z \\
    & =\left(\left(\nabla_X h\right)(Y, Z)-\left(\nabla_Y h\right)(X, Z)\right) V.
    \end{aligned}
    $$
    Since $V$ is non-vanishing, we have that $d_{\mathrm{KV}} \theta=0$ if and only if $\left(\nabla_X h\right)(Y, Z)=\left(\nabla_Y h\right)(X, Z)$ for all $X, Y, Z \in \Gamma(T M)$.
\end{proof}

Proposition \ref{proposition_4_3} and \ref{proposition_4_5} give 
characterizations of deformation from a flat connection arising from projective and dual-projective transformation in  
terms of KV cochains. 

\section{Calculations of Quantities Related to Hessian Geometry}

Throughout this section, we fix a Riemannian metric $g$ on $M$.

\begin{defin}\label{defin_5_1}
    The conjugate connection $\nabla^*$ of $\nabla$ is defined to be the unique affine connection on the tangent bundle $T M \rightarrow M$ such that the equation
    $$
    Z g(X, Y)=g\left(\nabla_Z X, Y\right)+g\left(X, \nabla_Z^* Y\right)
    $$
    holds for all $X, Y, Z \in \Gamma(T M)$.
\end{defin}

See for example \cite{7} for a proof of the fact that $\nabla^*$ is well-defined.
\begin{lemma}\label{lemma_5_2}
    If the Riemannian metric $g$ is Codazzi-coupled with $\nabla$, then the conjugate connection $\nabla^*$ is flat and torsion-free, and $\frac{1}{2}\left(\nabla^*+\nabla\right)$ is the Levi-Civita connection of $(M, g)$. 
\end{lemma}
\begin{proof}
    This statement is proved in \cite{8} and \cite{9}.
\end{proof}

Lemma\ref{lemma_5_2} describes what is known as Hessian geometry. (Recall that a Riemannian manifold $(M, g)$ together with a flat and torsion-free connection $\nabla$ is said to be of Hessian type, if the conjugate connection $\nabla^*$ of $\nabla$ is also flat and torsion-free, or equivalently, the Riemannian metric $g$ can be expressed locally as the second derivative of a smooth function on $M$.) 

From now on, we consider only Riemannian manifolds of Hessian type.

As usual, we denote by 
$$R\colon \Gamma(T M) \times \Gamma(T M) \times \Gamma(T M) \longrightarrow \Gamma(T M), \quad (X,Y,Z)\longmapsto R(X,Y)Z
$$
the Riemann curvature tensor of $(M, g)$.

\begin{theorem}\label{theorem_5_3}
    Suppose that $g$ is Codazzi-coupled with $\nabla$. Then, for any $X, Y, Z \in \Gamma(T M)$, it holds that $\left(d_{\mathrm{KV}} \nabla^*\right)(X, Y, Z)=4 R(X, Y) Z$.
\end{theorem}
\begin{proof}
    Define the difference tensor $\theta:=\nabla^*-\nabla$. Computation yields that
    $$
    \begin{aligned}
    \nabla_X(\theta(Y, Z))-\nabla_Y(\theta(X, Z))= & \left(\nabla_X \nabla_Y^* Z-\nabla_X \nabla_Y Z+\nabla_X^* \nabla_Y^* Z-\nabla_X^* \nabla_Y^* Z\right) \\
    & -\left(\nabla_Y \nabla_X^* Z-\nabla_Y \nabla_X Z+\nabla_Y^* \nabla_X^* Z-\nabla_Y^* \nabla_X^* Z\right) \\
    = & \nabla_{[X, Y]}^* Z-\nabla_{[X, Y]} Z-\theta\left(X, \nabla_Y^* Z\right)+\theta\left(Y, \nabla_X^* Z\right) \\
    = & \theta([X, Y], Z)+\left[\theta\left(Y, \nabla_X^* Z\right)-\theta\left(X, \nabla_Y^* Z\right)\right]
    \end{aligned}
    $$
    holds for all $X, Y, Z \in \Gamma(T M)$. Therefore for any $X, Y, Z \in \Gamma(T M)$ we have that
    $$
    \begin{aligned}
    \left(d_{\mathrm{KV}} \theta\right)(X, Y, Z)= & -\nabla_X(\theta(Y, Z))+\theta\left(\nabla_X Y, Z\right)+\theta\left(\nabla_X Z, Y\right) \\
    & +\nabla_Y(\theta(X, Z))-\theta\left(\nabla_Y X, Z\right)-\theta\left(\nabla_Y Z, X\right) \\
    = & \theta([X, Y], Z)+\left[\theta\left(Y, \nabla_X Z\right)-\theta\left(X, \nabla_Y Z\right)\right] \\
    & -\left[\nabla_X(\theta(Y, Z))+\nabla_Y(\theta(X, Z))\right] \\
    = & \theta(Y, \theta(X, Z))-\theta(X, \theta(Y, Z)) \\
    = & 4 R(X, Y) Z
    \end{aligned}
    $$
    by \cite{9}. Since by Proposition \ref{proposition_4_1} $d_{\mathrm{KV}} \nabla=0$, we have that $d_{\mathrm{KV}} \nabla^*=d_{\mathrm{KV}}\left(\nabla^*-\nabla\right)=d_{\mathrm{KV}} \theta$. This concludes the proof.
\end{proof}
\begin{corollary}\label{corollary_5_4}
    Suppose that $g$ is Codazzi-coupled with $\nabla$. Then $d_{\mathrm{KV}}R=0$.
\end{corollary}
\begin{proof}
    By Theorem \ref{theorem_5_3}, $R=d_{\mathrm{KV}} \frac{1}{4}\nabla^*$. By Lemma \ref{lemma_2_6}, $d_{\mathrm{KV}} R=\frac{1}{4}d_{\mathrm{KV}} d_{\mathrm{KV}} \nabla^*=0$.
\end{proof}

\section{An Example of Non-Vanishing KV Cohomology Group}

In this section, we construct a non-trivial example of second KV cohomology group. For simplicity, we consider the two dimensional case. Our procedure consists of two steps. First, we study conditions of the vanishing of KV differential $d_{\mathrm{KV}}\theta$ for a KV 2-chain $\theta$ that results from the conformal transformation of a flat manifold. Second, we construct an explicit example of which the above $\theta$ is not in the image of the KV differential $d_{\mathrm{KV}}$.

\begin{theorem}\label{theorem_5_5}
    Suppose that $(M,g)$ is a two dimensional flat Riemannian manifold with Levi-Civita connection $\nabla$.
    Let $f \in C^{\infty}(M)$ and let $\theta \in C^2(A)$ be the $\mathrm{KV}$ $2$-cochain satisfying 
    $$\theta(X, Y)=-g(X, Y) \operatorname{grad} f+\{(X f) Y+(Y f) X\}$$
    for all $X, Y \in \Gamma(T M)$. Then the Levi-Civita connection of $\left(M, e^{2 f} g\right)$ is $\nabla+\theta$, and the following statements are equivalent:
    \begin{enumerate}
        \item[$(i)$] $d_{\mathrm{KV}} \theta=0$;
        \item[$(ii)$] $\nabla+\theta$ is flat;
        \item[$(iii)$] $f$ is a harmonic function on $(M, g)$.
    \end{enumerate}
\end{theorem}
\begin{proof}
    Denote by $\Delta$ the Laplace-Beltrami operator on $(M, g)$ and denote $\tilde{g}:=e^{2 f} g$. The fact that the Levi-Civita connection of $(M, \tilde{g})$ is $\nabla+\theta$ follows from Koszul's formula. According to \cite{10}, the Riemann curvature tensor of $(M, \tilde{g})$ is 
    $$e^{2 f}(K+\Delta f) \tilde{g} \owedge \tilde{g}$$
    where $K$ is the Gaussian curvature of $(M, g)$. Since $\nabla$ is flat, we have that $K=0$. Therefore $\nabla+\theta$ is flat if and only if $\Delta f=0$. This proves that $(ii)$ is equivalent to $(iii)$.

    We shall now prove that $(i)$ is also equivalent to $(iii)$. For each $1$-form $\omega \in \Omega^{1}(M)$, let $\omega^{\#} \in \Gamma(T M)$ be the vector field satisfying $\omega(X)=g\left(\omega^{\#}, X\right)$ for all $X \in \Gamma(T M)$. Recall that for any $X \in \Gamma(T M)$ we have $\left(\nabla_X d f\right)^{\#}=\nabla_X \operatorname{grad} f$. Since $\nabla$ is compatible with $g$, the equality
    $$
    \begin{aligned}
    \left(d_{\mathrm{KV}} \theta\right)(X, Y, Z)= & \left[\left(\nabla_X g\right)(Y, Z)-\left(\nabla_Y g\right)(X, Z)\right] \operatorname{grad}f+g(Y, Z) \nabla_X \operatorname{grad}f-g(X, Z) \nabla_Y \operatorname{grad}f \\
    & +\left(\nabla_Y d f\right)(X) Z+\left(\nabla_Y d f\right)(Z) X-\left(\nabla_X d f\right)(Y) Z-\left(\nabla_X d f\right)(Z) Y\\ 
    =&\left(\nabla_Y d f\right)(Z) X-\left(\nabla_X d f\right)(Z) Y+\left[\left(\nabla_Y d f\right)(X)-\left(\nabla_X d f\right)(Y)\right] Z\\ 
    &+g(Y, Z)\left(\nabla_X d f\right)^{\#}-g(X, Z)\left(\nabla_Y d f\right)^{\#}
    \end{aligned}
    $$
    holds for all $X, Y, Z \in \Gamma(T M)$. By passing to trivializing neighbourhood if necessary, we can assume w.l.o.g. that there exists $e_1, e_2 \in \Gamma\left(T M\right)$ such that $g\left(e_i, e_j\right)=\delta_{i j}$ for all $i, j \in\{1,2\}$.

    Take distinct $i, j \in\{1,2\}$. By observation $\left(d_{\mathrm{KV}} \theta\right)\left(e_i, e_i, e_i\right)=\left(d_{\mathrm{KV}} \theta\right)\left(e_i, e_i, e_j\right)=0$. Since the Hessian $\nabla d f$ is symmetric, we obtain 
    $$g\left(\left(d_{\mathrm{KV}} \theta\right)\left(e_i, e_j, e_i\right), e_i\right)=2\left(\nabla_{e_j} d f\right)\left(e_i\right)-\left(\nabla_{e_i} d f\right)\left(e_j\right)-\left(\nabla_{e_j} d f\right)\left(e_i\right)=0.$$
    Also direct computation yields that $$g\left(\left(d_{\mathrm{KV}} \theta\right)\left(e_i, e_j, e_i\right), e_j\right)=-\left(\nabla_{e_i} d f\right)\left(e_i\right)-\left(\nabla_{e_j} d f\right)\left(e_j\right)=-\Delta f.$$
    Similarly we have $g\left(\left(d_{\mathrm{KV}} \theta\right)\left(e_i, e_j, e_j\right), e_j\right)=0$ and $g\left(\left(d_{\mathrm{KV}} \theta\right)\left(e_i, e_j, e_j\right), e_i\right)=\Delta f$. Therefore $d_{\mathrm{KV}} \theta=0$ if and only if $\Delta f=0$.
\end{proof}

Using Theorem \ref{theorem_5_5}, we now can construct an explicit example of which second KV cohomology does not vanish.

\begin{example}
Let $M:=\left\{(x, y) \in \mathbb{R}^2 \mid x^2+y^2 \neq 0\right\}$ be the punctured plane, and denote by $\Gamma(T M)$ the $\mathbb{R}$-vector space of $C^{\infty}$ vector fields on $M$. Let $g=d x \otimes d x+d y \otimes d y$ be the standard Euclidean metric on $M$, and denote by $\nabla$ the Levi-Civita connection of $(M, g)$. Then $\nabla$ is flat and torsion free. Denote by $A:=(\Gamma(T M), \nabla)$ the $\mathrm{KV}$ algebra of $\nabla$, and denote by $\left(\bigoplus_{i=0}^{\infty} C^i(A), d_{\mathrm{KV}}\right)$ its $\mathrm{KV}$ cochain complex. Consider the smooth function:
$$
f: M \longrightarrow \mathbb{R}, \quad(x, y) \longmapsto \frac{1}{2} \ln \left(x^2+y^2\right) .
$$
and consider the $\mathrm{KV}$ $2$-cochain $\theta \in C^2(A)$ satisfying $$\theta(X, Y)=-g(X, Y)\operatorname{grad}f+(X f) Y+(Y f) X$$ for all $X, Y \in \Gamma(T M)$. 
Since $f$ is a harmonic function on $(M, g)$, by Theorem \ref{theorem_5_5}, we have that $d_{\mathrm{KV}} \theta=0$. 

{\bf Claim}. The $\mathrm{KV}$ $2$-cochain $\theta$ is not an element of $\operatorname{Im}\left(d_{\mathrm{KV}}\right)$.

\begin{proof}
Assume the contrary, then by Theorem \ref{theorem_4_2} there exists $Z \in \Gamma(T M)$ such that the equation
\begin{equation}\tag{$*$}
    \nabla_{\nabla_X Y} Z-\nabla_X \nabla_Y Z=\theta(X, Y)
\end{equation}
holds for all $X, Y \in \Gamma(T M)$. Take $u, v \in C^{\infty}(M)$ such that $Z=u \frac{\partial}{\partial x}+v \frac{\partial}{\partial y}$. Then equation ($*$) implies
$$
\begin{aligned}
& -u_{x x} \frac{\partial}{\partial x}-v_{x x} \frac{\partial}{\partial y}=f_y \frac{\partial}{\partial y}-f_x \frac{\partial}{\partial x} \\
& -u_{x y} \frac{\partial}{\partial x}-v_{x y} \frac{\partial}{\partial y}=f_y \frac{\partial}{\partial x}+f_x \frac{\partial}{\partial y} \\
& -u_{y y} \frac{\partial}{\partial x}-v_{y y} \frac{\partial}{\partial y}=f_x \frac{\partial}{\partial x}-f_y \frac{\partial}{\partial y}
\end{aligned}
$$
and in particular, $u$ satisfies the system of partial differential equations on $M$:
$$
-\left(x^2+y^2\right) \operatorname{Hess}(u)=\left[\begin{array}{cc}
x & y \\
y & -x
\end{array}\right] .
$$
 Consider open subset $\Omega:=\left\{(x, y) \in \mathbb{R}^2 \mid y \neq 0\right\}$ of $M$. Direct computation yields that there exists $a, b, c \in \mathbb{R}$ such that 
$$u(x, y)=\dfrac{x}{2} \ln \left(x^2+y^2\right)+\arctan \left(\dfrac{x}{y}\right)y +a x+b y+c$$
for all $(x,y)\in\Omega$. In particular $u(2,t)=\operatorname{ln}(t^2+4)+\operatorname{arctan}(2/t)t+b\cdot t+(2a+c)$ for all $t\in \mathbb{R}^\times$. Differentiation yields that $u_y(2,t)=\operatorname{arctan}(2/t)+b$ for all $t\in \mathbb{R}^\times$. Since $u_y(2,t)\rightarrow \frac{\pi}{2}+b$ as $t\rightarrow 0+$ and $u_y(2,t)\rightarrow -\frac{\pi}{2}+b$ as $t\rightarrow 0-$, we obtain that the function 
$$\begin{aligned}
\mathbb{R}^\times & \longrightarrow \mathbb{R} \\
t & \longmapsto u_y(2,t)
\end{aligned}$$
cannot be extended continuously to $\mathbb{R}$. Therefore $u|_\Omega$ cannot be extended to a smooth function defined on $M$, contradiction. This proves the claim.
\end{proof}

That $d_{\mathrm{KV}} \theta=0$ but $\theta \notin \operatorname{Im}\left(d_{\mathrm{KV}}\right)$ means that
the second cohomology group of $\left(\bigoplus_{i=0}^{\infty} C^i(A), d_{\mathrm{KV}}\right)$ does not vanish.
\end{example}

\section*{Appendix}

Let $\pi\colon E \rightarrow M$ be a smooth $\mathbb{R}$-vector bundle over a differentiable manifold $M$. Denote by $\Gamma(E)$ the space of $C^{\infty}$ sections of $\pi$, and $\Gamma(\text{End} E)$ the space of $C^{\infty}$ sections of the endomorphism bundle of $\pi$. For each $k \in \mathbb{N}$, denote by $\Omega^k\left(M; E\right):=\Omega^k(M) \otimes \Gamma(E)$ the $C^{\infty}(M)$-module of vector-valued $k$-forms on $M$. Let $\nabla$ be a connection on $E \xrightarrow{\pi} M$, i.e. $\nabla\colon \Gamma(E) \rightarrow \Omega^1(M) \otimes \Gamma(E)$ is an $\mathbb{R}$-linear transform such that the Leibniz rule
$$
\nabla(f \cdot s)=d f \otimes s+f \nabla s
$$
holds for all $f \in C^{\infty}(M)$ and $s \in \Gamma(E)$. For a given smooth vector field $X$ on $M$, we have the mapping
$$
\begin{aligned}
\nabla_X\colon \Gamma(E) & \longrightarrow \Gamma(E) \\
s & \longmapsto(\nabla s)(X)
\end{aligned}
$$
induced from $\nabla$. The differential operator $\nabla$ applicable to $\Gamma(E)$ can be extended to a map
$$
d^{\nabla}\colon \Omega^k(M; E) \rightarrow \Omega^{k+1}(M ; E)
$$
for every $k \in \mathbb{N}$, via the formula
$$
d^{\nabla}(\omega \otimes s)=d \omega \otimes s+(-1)^k \omega \wedge \nabla s
$$
for any $\omega \in \Omega^k(M)$ and $s \in \Gamma(E)$.

Straightforward computation yields that
$$
\begin{aligned}
    \left(d^{\nabla} \theta\right)\left(X_0, \ldots, X_k\right)=&\sum\limits_{i=0}^k(-1)^i \nabla_{X_i}\left(\theta\left(X_0, \ldots, \hat{X}_i, \ldots, X_k\right)\right) \\ 
    & +\sum\limits_{0 \leqslant i<j \leqslant k}(-1)^{i+j} \theta\left(\left[X_i, X_j\right], X_0, \ldots, \hat{X}_i, \ldots, \hat{X}_j, \ldots, X_k\right)
\end{aligned}
$$
for all $\theta \in \Omega^k(M; E)$ and all smooth vector fields $X_0, \ldots, X_k$ on $M$, where the hats on $\hat{X}_i$ and $\hat{X}_j$ indicate that the $X_i$ and $X_j$ terms are omitted. In particular, for $k=2$ the equation above simplifies to
$$
\left(d^{\nabla} \theta\right)(X, Y)=\nabla_X(\theta(Y))-\nabla_Y(\theta(X))-\theta([X, Y])
$$
for all $\theta \in \Omega^1(M; E)$ and all smooth vector fields $X, Y$ on $M$.

By the general theory of connections \cite{10}, there exists $R^{\nabla} \in \Omega^2(M) \otimes \Gamma(\text{End}E)$ such that $d^{\nabla}\left(d^{\nabla} \theta\right)=R^{\nabla} \wedge \theta$ for all $k \in \mathbb{N}$ and $\theta \in \Omega^k(M; E)$. The matrix-valued 2-form $R^{\nabla}$ is termed the curvature of $\nabla$. To compute $R^{\nabla}$ explicitly, it suffices to notice that
$$
\begin{aligned}
R^{\nabla}(X, Y) s & =\left(d^{\nabla}\left(d^{\nabla} s\right)\right)(X, Y) \\
& =(d^{\nabla}(\nabla s))(X, Y) \\
& =\nabla_X((\nabla s)(Y))-\nabla_Y((\nabla s)(X))-(\nabla s)([X, Y]) \\
& =\nabla_X\left(\nabla_Y s\right)-\nabla_Y\left(\nabla_X s\right)-\nabla_{[X, Y]} s
\end{aligned}
$$
holds for all $s \in \Gamma(E)$ and all smooth vector fields $X, Y$ on $M$. Therefore
$$
R^{\nabla}(X, Y)=\left[\nabla_X, \nabla_Y\right]-\nabla_{[X, Y]}
$$
for all smooth vector fields $X, Y$ on $M$.

The connection $\nabla$ is said to be flat, if its curvature $R^{\nabla}$ vanishes identically. If $\nabla$ is flat, then $d^{\nabla} \circ d^{\nabla}=0$, and hence $\left(\bigoplus_{i=0}^{\infty} \Omega^i(M; E), d^{\nabla}\right)$ is a differential graded $\mathbb{R}$-vector space, termed the de Rham complex of $M$ twisted by $\nabla$.

From now on, we shall consider only the case that $\pi\colon E \rightarrow M$ is the tangent bundle of $M$, and $\nabla$ is flat and torsion-free. As usual, we denote by $A$ the $\mathrm{KV}$ algebra of $\nabla$, and denote by $\left(\bigoplus_{i=0}^{\infty} C^i(A), d_{\mathrm{KV}}\right)$ its $\mathrm{KV}$ cochin complex. Since $\Omega^k(M; E) \subseteq C^k(A)$ for all $k \in \mathbb{N}$, it is natural to compare the cohomology groups of $\left(\bigoplus_{i=0}^{\infty} \Omega^i(M; E), d^{\nabla}\right)$ and $\left(\bigoplus_{i=0}^{\infty} C^i(A), d_{\mathrm{KV}}\right)$. However, it turns out that in general the cohomology of $\left(\bigoplus_{i=0}^{\infty} \Omega^i(M ; E), d^{\nabla}\right)$ cannot be embedded into the cohomology of $\left(\bigoplus_{i=0}^{\infty} C^i(A), d_{\mathrm{KV}}\right)$ as a graded abelian group, and vice versa.

An example can be constructed as follows:

Let $M:=\left\{(x, y) \in \mathbb{R}^2 \mid x^2+y^2 \neq 0\right\}$ be the punctured plane. Since the tangent bundle $TM \rightarrow M$ is trivial, we have that $\Omega^k\left(M; TM\right)=\Omega^k(M) \oplus \Omega^k(M)$ for all $k \in \mathbb{N}$. Let $g=d x \otimes d x+d y \otimes d y$ be the standard Euclidean metric on $M$, and let $\nabla$ be the Levi-Civita connection of $(M, g)$. Then $d^{\nabla}=d \oplus d$, and hence the cohomology of $\left(\bigoplus_{i=0}^{\infty} \Omega^i(M; TM), d^{\nabla}\right)$ is $\bigoplus_{i=0}^{\infty} H^i(M; \mathbb{R})^{\oplus 2}$. Since $M$ is homotopic to the unit circle $\mathbb{S}^1$, we conclude that $H^0(M; \mathbb{R})=\mathbb{R}$ and $H^2(M; \mathbb{R})=0$. Therefore the second cohomology of $\left(\bigoplus_{i=0}^{\infty} \Omega^i(M; TM), d^{\nabla}\right)$ vanishes and the zeroth cohomology of $\left(\bigoplus_{i=0}^{\infty} \Omega^i(M; TM), d^{\nabla}\right)$ does not vanish.

Let $A$ be the $\mathrm{KV}$ algebra of $\nabla$, and $\left(\bigoplus_{i=0}^{\infty} C^i(A), d_{\mathrm{KV}}\right)$ its $\mathrm{KV}$ cohain complex. By Example 6.2 we have that the second cohomology of $\left(\bigoplus_{i=0}^{\infty} C^i(A), d_{\mathrm{KV}}\right)$ does not vanish. We claim that the zeroth cohomology of $\left(\bigoplus_{i=0}^{\infty} C^i(A), d_{\mathrm{KV}}\right)$ vanishes. Indeed, it suffices to prove the following lemma:

\begin{lemma}
    Let $\Omega$ be a domain in $\mathbb{R}^n$, and let $X$ be a smooth vector field on $\Omega$. If $[X, Y]=0$ for every smooth vector field $Y$ on $\Omega$, then $X=0$.
\end{lemma}
\begin{proof}
    Let $x^1, \ldots, x^n$ be the standard coordinate system on $\Omega$. Then there exists $X^1, \ldots, X^n \in C^{\infty}(\Omega)$ such that $X=X^i \frac{\partial}{\partial x^i}$. Since
    $$
    \dfrac{\partial X^i}{\partial x^j} \dfrac{\partial}{\partial x^i}=\left[X^i \dfrac{\partial}{\partial x^i},-\dfrac{\partial}{\partial x^j}\right]=0
    $$
    holds for all $j \in\{1, \ldots, n\}$, we have that $X^1, \ldots, X^n$ are constant functions. Therefore
    $$
    \begin{aligned}
    X & =X^i \dfrac{\partial}{\partial x^i} \\
    & =X^i \dfrac{\partial x^j}{\partial x^i} \dfrac{\partial}{\partial x^j} \\
    & =X^i\left[\dfrac{\partial}{\partial x^i}, x^j \dfrac{\partial}{\partial x^j}\right] \\
    & =\left[X^i\dfrac{\partial}{\partial x^i}, x^j \dfrac{\partial}{\partial x^j}\right] \\
    & =0
    \end{aligned}
    $$
    as desired.
\end{proof}

The discussion above shows that in general KV cohomology and de Rham cohomology (twisted
by a local system) are very different from each other.

\quad

\section*{Statements and Declarations}
The authors did not receive support from organizations for the submitted work. All authors certify that they have no affiliations with or involvement in any organization or entity with any financial interest or non-financial interest in the subject matter or materials discussed in this manuscript.

\quad

\vspace{50pt}

Hanwen Liu

\noindent \textsc{Mathematics Institute, University of Warwick, Coventry, CV4 7AL, UK} 

\noindent \textit{E-mail address:} \href{mailto:hanwen.liu@warwick.ac.uk}{\nolinkurl{hanwen.liu@warwick.ac.uk}}
\bigskip

\quad

Jun Zhang

\noindent \textsc{University of Michigan, Ann Arbor, MI 48109, USA} 

\noindent \textit{E-mail address:} \href{mailto:junz@umich.edu}{\nolinkurl{junz@umich.edu}}
\bigskip

\noindent \textsc{Shanghai Institute for Mathematics and Interdisciplinary Sciences, Shanghai, China} 

\noindent \textit{E-mail address:} \href{mailto:junz@simis.cn}{\nolinkurl{junz@simis.cn}}
\bigskip

\end{document}